\documentclass[12pt,draft,reqno]{amsart}

\usepackage{amsmath,amssymb,amscd,amsfonts,amsthm}
\usepackage{graphics}
\usepackage{dsfont}
\usepackage[all]{xy}
\usepackage{enumerate}
\usepackage{dsfont}
\usepackage{bbold}
\usepackage{mathbbol}

\usepackage{array,amscd,latexsym, verbatim}

{\catcode`\@=11
\gdef\n@te#1#2{\leavevmode\vadjust{%
 {\setbox\z@\hbox to\z@{\strut#1}%
  \setbox\z@\hbox{\raise\dp\strutbox\box\z@}\ht\z@=\z@\dp\z@=\z@%
  #2\box\z@}}}
\gdef\leftnote#1{\n@te{\hss#1\quad}{}}
\gdef\rightnote#1{\n@te{\quad\kern-\leftskip#1\hss}{\moveright\hsize}}
\gdef\?{\FN@\qumark}
\gdef\qumark{\ifx\next"\DN@"##1"{\leftnote{\rm##1}}\else
 \DN@{\leftnote{\rm??}}\fi{\rm??}\next@}}
\DeclareFontFamily{OT1}{wncyr}{\hyphenchar\font45 }
\DeclareFontShape{OT1}{wncyr}{m}{n}{%
   <5> <6> <7> <8> <9> gen * wncyr
   <10> <10.95> <12> <14.4> <17.28> <20.74>  <24.88>wncyr10}{}
\DeclareFontShape{OT1}{wncyr}{m}{it}{%
   <5> <6> <7> <8> <9> gen * wncyi
   <10> <10.95> <12> <14.4> <17.28> <20.74> <24.88> wncyi10}{}
\DeclareFontShape{OT1}{wncyr}{m}{sc}{%
   <5> <6> <7> <8> <9> <10> <10.95> <12> <14.4>
   <17.28> <20.74> <24.88>wncysc10}{}
\DeclareFontShape{OT1}{wncyr}{b}{n}{%
   <5> <6> <7> <8> <9> gen * wncyb
   <10> <10.95> <12> <14.4> <17.28> <20.74> <24.88>wncyb10}{}
\input cyracc.def
\def\rus{\usefont{OT1}{wncyr}{m}{n}\cyracc\fontsize{9}{11pt}\selectfont}

\DeclareMathSizes{9}{9}{7}{5}

\DeclareFontFamily{OT1}{wncyr}{\hyphenchar\font45 }
\DeclareFontShape{OT1}{wncyr}{m}{n}{%
   <5> <6> <7> <8> <9> gen * wncyr
   <10> <10.95> <12> <14.4> <17.28> <20.74>  <24.88>wncyr10}{}
\DeclareFontShape{OT1}{wncyr}{m}{it}{%
   <5> <6> <7> <8> <9> gen * wncyi
   <10> <10.95> <12> <14.4> <17.28> <20.74> <24.88> wncyi10}{}
\DeclareFontShape{OT1}{wncyr}{m}{sc}{%
   <5> <6> <7> <8> <9> <10> <10.95> <12> <14.4>
   <17.28> <20.74> <24.88>wncysc10}{}
\DeclareFontShape{OT1}{wncyr}{b}{n}{%
   <5> <6> <7> <8> <9> gen * wncyb
   <10> <10.95> <12> <14.4> <17.28> <20.74> <24.88>wncyb10}{}
\input cyracc.def
\def\rus{\usefont{OT1}{wncyr}{m}{n}\cyracc\fontsize{9}{11pt}\selectfont}

\DeclareMathSizes{9}{9}{7}{5} 

\theoremstyle{plain}

\newtheorem{theorem}{Theorem}

\newtheorem{lemma}{Lemma}
\newtheorem{corollary}{Corollary}

\theoremstyle{definition}

\newtheorem*{remarknonumber}{\it Remark}

\newtheorem{nothing*}[theorem]{}
\newtheorem{subnothing*}[sub]{}

\theoremstyle{remark}

\def\bAd {{\mathbf A}\!^d}

\def\fratop{\genfrac{}{}{0pt}1}

\newcommand{\dss}{\hskip -2mm\rotatebox{68}{\raisebox{-1.8\height}{\mbox{\normalsize -\hskip .1mm-\hskip .1mm-}}}\hskip -.6mm}

\begin{document}

\title[Generic algebras]
{Generic algebras:\\ rational parametrization\\ and normal forms}

\author[Vladimir L. Popov]{Vladimir L. Popov${}^*$}
\address{Steklov Mathematical Institute,
Russian Academy of Sciences, Gub\-kina 8, Moscow
119991, Russia}

\address{National Research University\\ Higher School of Economics, Myas\-nitskaya
20, Moscow 101000,\;Russia} \email{popovvl@mi.ras.ru}

\thanks{
 ${}^*$\,Supported by
 grant {\rus RFFI
14-01-00160}.}

\begin{abstract} For every algebraically closed field  $\boldsymbol k$ of characteri\-stic different from  $2$, we prove the following:

(1)  Generic finite dimensional (not necessarily associative)
$\boldsymbol k$-algebras of a fixed dimension, considered up to isomorphism, are parametrized
by the values of a tuple of algebraically independent over $\boldsymbol k$ rational functions in the structure constants.

(2) There exists an ``algebraic normal form'', to which the set of structure constants of every such algebra can be uniquely trans\-formed
by means of passing to
its new basis, namely: there are two finite systems of nonconstant polynomials on
the space of structure constants, $\{f_i\}_{i\in I}$ and $\{b_j\}_{j\in J}$, such that
the ideal generated by the set $\{f_i\}_{i\in I}$ is prime and, for every tuple $c$ of structure constants satisfying the property
 $b_j(c)\neq 0$ for all $j\in J$, there exists a unique new basis of this algebra in which the tuple
$c'$ of its structure constants satisfies the property $f_i(c')=0$ for all $i\in I$.
\end{abstract}

\maketitle

\section{Introduction}
 Hereinafter $\boldsymbol k$ denotes an algebraically closed field  
 of arbitrary charac\-teristic different from  $2$.\;Let
 $V$ be a finite dimensional vector space over the field
$\boldsymbol k$. We put
$$n:=\dim_{\boldsymbol k} V,\qquad \mathcal M:=V^*\otimes V^*\otimes V.$$

Putting in correspondence to an element $\sum \ell\otimes \ell'\otimes v\in\mathcal M$ the bilinear map
 \begin{equation*}
 V\times V\to V,\quad (a, b)\mapsto \textstyle\sum \ell(a)\ell'(b)v,
  \end{equation*}
  we obtain a well-defined bijection between
   $\mathcal M$ and the set of all
   bilinear maps и  $ V\times V\to V$.\;This bijections commutes with the natural action
   of the group
 \begin{equation*}
G:={\rm GL}(V)
\end{equation*}
on $\mathcal M$ and its action on the the set of all bilinear maps $ V\times V\to V$
given by the formula $(g\cdot \varphi) (a, b)=g\cdot (\varphi(g^{-1}\cdot a, g^{-1}\cdot b))$.

The problem of classification of $G$-orbits in $\mathcal M$ is linked with the applications in the theory of continuous and discrete dynamical sys\-tems; see \cite{Bo}, \cite{DM}, \cite{MM}.\;In these papers, the case $n=2$ is considered (under the restriction ${\rm char}\,\boldsymbol k\neq 3$ supplementary to ${\rm char}\,\boldsymbol k\neq 2$): in \cite{Bo}, \cite{DM} symmetric, and in
\cite{MM} arbitrary generic maps are considered.\;The first main result of paper \cite{MM} is the proof of the existence of
\begin{enumerate}[\hskip 2.0mm\rm(i)]
\item a nonempty open subset  $U$ in $\mathcal M$ and
\item four
$G$-invariant rational functions $f_1$, $f_2$, $f_3$, $f_4$ on $\mathcal M$
defined everywhere on $U$
\end{enumerate}
such that every two tensors $a, b\in U$ lie in the same
$G$-orbit if and only if $f_i(a)=f_i(b)$ for all $i$.\;Thus generic $G$-orbits in $\mathcal M$
are bijectively parametrized by points of the image of rational map
\begin{equation}\label{moduli}
U\to {\bf A}\!^4,\quad m\mapsto (f_1(m), f_2(m), f_3(m), f_4(m)).
\end{equation}

The second main result of paper \cite{MM} is the proof of the existence of a  ``normal form'' to which
 every tensor in $U$ can be uniquely transformed by means of an element of $G$; the set of all normal forms is algebraic, i.e., has the appearance $U\cap \{m\in \mathcal M\mid h_i(m)=0\;\forall\;i\in I\}$ for some
 finite set $\{h_i\}_{i\in I}$ of polynomial functions on $\mathcal M$. In \cite{MM}, the set
 $U$, the functions $f_i$, the normal forms, and the elements of $G$ transforming tensors into normal forms are explicitly given by bulky formulas, and the proofs are largely based on cumbersome explicit calculations.

  According to the classical Rosenlicht theorem \cite[Thm.\,2]{Ros}, for any action of an algebraic group on an irreducible algebraic variety, orbits of points in general position are separated by a finite tuple of invariant rational functions.\;Therefore,
  leaving aside the explicit formulas, only the claim about the number of separating invariants
  ($4$, what is equal to
  $\dim\mathcal M-\dim G$) does not follow
   from Rosenlicht's theorem. This claim has the following important complements
   (a), (b), and (c), which are not noted in
\cite{MM}:

\vskip 1mm

(a) The field ${\boldsymbol k}(\mathcal M)^G$ of all $G$-invariant rational functions on $\mathcal M$
is a finite purely inseparable extension of the field ${\boldsymbol k}(f_1,\ldots, f_4)$, and for
${\rm char}\,{\boldsymbol k}=0$ these fields coincide
(this follows from \cite[Sect.\;18.2,\;Thm.]{Bor} because ${\boldsymbol k}(f_1,\ldots, f_4)$ separates
$G$-orbits of points of an open subset of $\mathcal M$).

 (b) By \cite[Prop.\;4.1]{MM}, in
  $\mathcal M$ there are points with finite
  $G$-stabilizers. Therefore, $\underset{m\in \mathcal M}{\max}\,\dim G\!\cdot\! m\!=\!\dim\,G\!=\!4$.
 Since the transcendence degree of  ${\boldsymbol k}(\mathcal M)^G$ over $\boldsymbol k$ is
  $\dim\,{\mathcal M}-\underset{m\in \mathcal M}{\max}\,\dim G\cdot m$ (see\;\cite[Sect.\;2.3, Cor.]{PV}), this means that its is equal to 4.

   (c)   The functions $f_1,\ldots, f_4$ are algebraically independent over $\boldsymbol k$,
   and therefore the image of map \eqref{moduli}
  contains a nonempty open subset of
    ${\bf A}\!^4$
   (this follows from (a) and (b)).

Property (c) is specific for the considered action: given the negative solution to the Noether problem \cite{Sa},
 for a linear action of an algebraic group, in general, it is impossible to separate orbits of points in general position by an algebraically independent system of rational invariants.

Here we show that both main results of paper \cite{MM},
together with complements (a)  and (c),
hold in a strengthened form in {\it arbitrary} dimension $n$.
Namely, we prove the following statements
(part of them is intentionally formulated in the form close to
the applications oriented papers
 \cite{Bo}, \cite{DM}, \cite{MM}).

\begin{theorem}[{{Rationality of the field of $G$-invariants of $\mathcal M$}}] \label{rat}
The field
$\boldsymbol k(\mathcal M)^G$ is a rational extension of the field $\boldsymbol k$ of transcendental degree $n^3-n^2$.
\end{theorem}

Theorem \ref{rat} and Rosenlicht's theorem \cite[Thm.\,2]{Ros} imply the following statement
(its part concerning the case  $n=1$ is clear):

\begin{corollary}[{{\rm $G$-equivalence of points in general position in $\mathcal M$}}]\label{cor1}
  Let $n\geqslant 2$.\;There exist a set $\{f_i\}_{i\in I}$ of $n^3-n^2$ algebraically independent over $\boldsymbol k$ rational $G$-invariant functions on $\mathcal M$, and a finite set $\{h_j\}_{j\in J}$ of nonconstant polynomial functions on $\mathcal M$ such that:
  \begin{enumerate}[\hskip 2.2mm\rm(i)]
\item for any points $a$ and $b$ of a dense open subset   $\{m\!\in\! \mathcal M\mid h_j(m)\!\neq\! 0\;\forall j\!\in\! J\}$ of $\mathcal M$ the conditions
    \begin{enumerate}[\hskip 4mm$\rm (a)$]
    \item
    $G\cdot a=G\cdot b$,
    \item
    $f_i(a)=f_i(b)$ for all $i\in I$
    \end{enumerate}
    are equivalent;
\item the set $\{f_i\}_{i\in I}$ generates the field $\boldsymbol k(\mathcal M)^G$
over  $\boldsymbol k$.
\end{enumerate}
If $n=1$, then  $\boldsymbol k(\mathcal M)^G=\boldsymbol k$, and
$\mathcal M\setminus \{0\}$ is a single $G$-orbit.
\end{corollary}

\begin{theorem}[{{\rm The existence of normal forms for $\mathcal M$}}]\label{section} There exist two
finite sets
 $\{s_p\}_{p\in P}$ and $\{t_q\}_{q\in Q}$ of nonconstant polynomial functions on
 $\mathcal M$, such that:
\begin{enumerate}[\hskip 2.2mm\rm(i)]
\item the closed subset $\mathcal S\!:=\!\{m\!\in\! \mathcal M\mid s_p(m)\!=\!0\,\forall p\!\in\! P\}$ is irreducible, and the map $\boldsymbol k(\mathcal M)^G\to \boldsymbol k(\mathcal S)$, $f\mapsto f|_{\mathcal S}$ is a well-defined $\boldsymbol k$-isomorphism of fields;
    \item
 for every point $a$ of the open dense subset $\{m\!\in\! \mathcal M\mid t_q(m)\!\neq\! 0\;\forall q\!\in\! Q\}$ of $\mathcal M$, there exists a unique element $g\in G$, for which
$g\cdot a\in\mathcal S$.
\end{enumerate}
\end{theorem}

Actually, we use another interpretation of elements of $\mathcal M$, consi\-de\-ring them as $\boldsymbol k$-algebra structures
(not necessarily associative) on $V$ in the sense of Bourbaki, i.e., ring structures, for which  $\boldsymbol k$ is the domain of external operators; see\;\cite{Bour}, \cite{Sch}.\;Fixing such a structure on $V$ is equivalent to fixing a bilinear map $\varphi\colon V\times V\to V$ that defines the product of elements  $a, b\in V$ by formula $ab:=\varphi(a, b)$.\;This gives a bijection between  $\mathcal M$ and the set of all such structures, which assigns to an element
 $m=\sum \ell\otimes \ell'\otimes v\in \mathcal M$ the $\boldsymbol k$-algebra structure on $V$ such that
 the product of elements $a, b\in V$ is defined by the formula
 \begin{equation}
\label{mult}
ab:=\sum \ell(a)\ell'(b)v.
\end{equation}
 We denote this $\boldsymbol k$-algebra by  $\{V, m\}$ and call  elements of the space  $\mathcal M$ {\it multiplications}.
For every element $g\in G$ and every $\boldsymbol k$-algebra $\{V, m\}$, the map $\{V, m\}\to \{V, g\cdot m\}$, $a\mapsto g\cdot a$ is an isomorphism of $\boldsymbol k$-algebras and every isomorphism between  $\boldsymbol k$-algebras is obtained in this way. Thus,
 $\boldsymbol k$-algebras $\{V, a\}$ and $\{V, b\}$ are isomorphic if and only if
 $G\cdot a= G\cdot b$, and the automorphism group of the
$\boldsymbol k$-algebra
$\{V, m\}$ is the $G$-stabilizer of the multiplication $m\in \mathcal M$.

  Let $\theta\in{\rm GL}(\mathcal M)$ be the involution induced by permuting the first two factors in $V^*\otimes V^*\otimes V$:
\begin{equation*}
\theta\colon\mathcal M\to \mathcal M,\quad \ell\otimes\ell'\otimes v\mapsto \ell'\otimes\ell\otimes v.
\end{equation*}
Since ${\rm char}\,{\boldsymbol k}\neq 2$, every multiplication $m\in \mathcal M$ can be written in the form
$m=\tfrac{1}{2}\big(m+\theta(m)\big)+\tfrac{1}{2}\big(m-\theta(m)\big)$, which implies that
\begin{equation}
\left.\begin{array}{ll}
\mathcal M&\hskip -2mm=\mathcal C\oplus\mathcal A,\quad\mbox{где}\\[3pt]
\mathcal C:&\hskip -2mm=\{m\in\mathcal M\mid \theta(m)=m\},\\[3pt]
\mathcal A:&\hskip -2mm=\{m\in\mathcal M\mid \theta(m)=-m\}.
\end{array}\right\}
\label{decomp}
\end{equation}
 The algebra $\{V, m\}$ is commutative, i.e.,\;$ab=ba$ in \eqref{mult} (respectively, anticommutative, i.e.,\;$ab=-ba$ в \eqref{mult}) if and only if
$m\in \mathcal C$ (respectively, $m\in \mathcal A$).\;The subspaces $\mathcal C$ and $\mathcal A$ in
$\mathcal M$ are $G$-invariant.

The proofs of Theorems \ref{rat} and \ref{section} are based on the triviality claims in the following two theorems:

\begin{theorem}\label{C} For points
   $m$ in general position in  $\mathcal C$, the algebras $\{V, m\}$ are simple commutative algebras with trivial automorphism group.
\end{theorem}

The analogous statement holds for  $\mathcal M$:

 \begin{theorem}\label{M}
    For points $m$ in general position in $\mathcal M$, the algebras  $\{V, m\}$ are simple and have trivial automorphism group.
\end{theorem}

Theorems  \ref{rat}, \ref{section}, and \ref{M} imply

\begin{corollary} \label{param1}
There exists a system of simple pairwise nonisomorphic  $n$-dimen\-sio\-nal
$\boldsymbol k$-algebras rationally parametrized by
$n^3-n^2$ parameters algebraically independent over $\boldsymbol k$.
\end{corollary}

We do not know whether the field  $\boldsymbol k(\mathcal C)^G$ is rational over  $\boldsymbol k$
or not, however, Theorem \ref{C} implies the following
\begin{corollary} \label{sr}
The field
$\boldsymbol k(\mathcal C)^G$ is a stably rational extension of the field  $\boldsymbol k$ of the transcendence degree $(n-1)n^2/2$.
\end{corollary}

Using Theorem \ref{C}, one proves the analogue of  Theorem \ref{section} for $\mathcal C$:

\begin{theorem}[{{\rm The existence of normal forms for $\mathcal C$}}]\label{sectionC}
The formulation of Theorem  {\rm \ref{section}} holds true if $\mathcal M$ in it  is replaced by $\mathcal C$.
\end{theorem}

  The following
 analogue of Corollary \ref{param1} follows from Theorems \ref{C}, \ref{sectionC} and Corollary \ref{sr}:

 \begin{corollary}\label{param2}
There exists a system of simple pairwise nonisomorphic  commutative $n$-dimensional
$\boldsymbol k$-algebras rationally parametrized by  po\-ints of
an $(n-1)n^2/2$-dimensional stably rational algebraic variety.
 \end{corollary}

In the proof of Theorem \ref{rat}, we use the following general statement:

 \begin{lemma} \label{123} Let  $L_1$, $L_2$ и $L_3$ be finite dimensional vector spaces over  $\boldsymbol k$, each of which is endowed with a linear action of an affine algebraic group $H$.\;Assume that
 \begin{enumerate}[\hskip 4.2mm\rm(i)]
 \item the actions of the group $H$ on $L_1$ and $L_2$ are locally free\footnote{see the definition of a locally free action at the end of this section.};
 \item the field ${\boldsymbol k}(L_1)^H$ is rational  over $\boldsymbol k$;
 \item $\dim L_3\geqslant \dim L_1$.
 \end{enumerate}
 Then the field ${\boldsymbol k}(L_2\oplus L_3)^H$ is also rational over $\boldsymbol k$.
 \end{lemma}

 Theorems \ref{C} and \ref{M} are proved in Section 2.  Lemma \ref{123}, Theorems
 \ref{rat}, \ref{section},
 \ref{sectionC}, and Corollary
\ref{sr} are proved in Section 3.

\vskip 3mm

 \noindent {\it Terminology and notation}

 \vskip 1mm

In what follows, variety means algebraic variety over
$\boldsymbol k$ in the sense of Serre.

The topological terms are related to the Zariski topology.

  If $K/L$ is a field extension, then  $K$ is called  {\it rational}\footnote{{\it pure} in the terminology of  \cite{Bour2}.} over $L$ if either  $K$ is a finitely generated purely transcendental extension of the field $L$ or
  $K=L$. If there exists a field extension $S/K$ such that  $S$ is rational over $K$ and over $L$, then  $K$  is called {\it stably rational} over $L$.

  For every nonempty set $T$ of transformations of a set $M$, we denote by $M^T$ the set of all joint fixed points of all transformation from $T$.

We freely use the standard notation, terminology, and conventions of the theory of algebraic groups and invariant theory \cite{Bor}, \cite{PV}, \cite{St}.\;In particular, we say that a certain property holds for
   {\it points in general position} in variety $X$ if it holds for all points of an open dense subset of $X$ (depending on the property under consideration).

   Action of an algebraic group
$H$ on a variety  $X$ means regular action, i.e., such that the map
$H\times X\to X$, $(h, x)\mapsto h\cdot x$ defining this action is a morphism.

If $X$ is irreducible, then $\pi_{H, X}\colon X\dashrightarrow X\dss H$ denotes a rational quotient for this action, i.e.,\;$X\dss H$ is a variety (defined up to a birational isomorphism) and  $\pi_{H, X}$ is a dominant rational map such that $\pi_{H, X}^*({\boldsymbol k}(X\dss H)={\boldsymbol k}(X)$.

An action of an algebraic group $H$ on a variety $X$ is called
{\it locally free} if, for points $x$ in general position in $X$, the orbit morphism $H\to H\cdot x$, $h\mapsto h\cdot x$ is an isomorphism. The latter condition is equivalent to the requirement that the following two conditions hold:
\begin{enumerate}[\hskip 4.2mm \rm(a)]
\item the specified orbit morphism is separable;
\item the stabilizer $H_x$ of the point $x$ is trivial.
\end{enumerate}
If ${\rm char}\,\boldsymbol k=0$, then condition (a) is automatically satisfied, i.e., in this case local freeness of an action is the same as triviality of the stabilizers of points in general position. If ${\rm char}\,\boldsymbol k>0$, this is not so.

\section{Proofs of Theorems \ref{C} and \ref{M}}

It will be convenient to fix a basis  $e_1,\ldots, e_n$ in $V$ and assume   that $G={\rm GL}_n(\boldsymbol k)$ (indenti\-fying elements of  $G$ with their
matrices in this basis).

 Let $\ell^1,\ldots, \ell^n$ be the dual basis in $V^*$.\;Then the elements
\begin{equation}\label{basis}
c^{pq}_r:=(\ell^p\otimes\ell^q+\ell^q\otimes\ell^p)\otimes e_r\in \mathcal C,\;\;\mbox{where $1\leqslant p\leqslant q\leqslant n$, $1\leqslant r\leqslant n,$}
\end{equation}
(somehow ordered) constitute a basis in $\mathcal C$, and so
\begin{equation}\label{dimC}
\dim \mathcal C=\frac{n^2(n+1)}{2}.
\end{equation}
From \eqref{dimC}, \eqref{decomp} and $\dim \mathcal M=n^3$  we get
\begin{equation}\label{dimA}
\dim\mathcal A=\frac{(n-1)n^2}{2}.
\end{equation}

Below we use the following known (see,\;e.g.,\,\cite[Lemma\;3]{AP}) statement:
\begin{lemma}\label{tran}
Let an algebraic group $H$ act on a variety $X$ and let $Y$ be a closed subset of $X$.\;Then $N_H(Y):=\{h\in H \mid h\cdot Y\subseteq Y\}$ is a closed subgroup of $H$ and
dimension of the closure in $X$ of the set  $\dim H\cdot Y$ is at most $\dim H+\dim Y-\dim N_H(Y)$.
\end{lemma}
\begin{proof} Regarding the closedness see \cite[Prop.\,I, 1,7]{Bor}. The inequality
is proved by apply\-ing
the fiber dimension theorem to the morphism
$H\times Y\to X$, $(h, y)\mapsto h\cdot y$.
\end{proof}

\begin{proof}[Proof of the claim of simplicity in Theorem {\rm \ref{C}}]\

Here, mutatis mutandis, the same arguments are applicable as in the  proof of simplicity of $\{V, m\}$  for a point $m$ in general position in $\mathcal M$, given in \cite[p.\;129]{Po}.
   For $n=1$ the claim is clear, therefore
   let  $n\geqslant 2$.

  For every $d\in \mathbf Z$, $0< d\leqslant n$, put
  \begin{equation*}
  \mathcal I_d:=\{m\in\mathcal C\mid \{V, m\} \;\;\mbox{has a $d$-dimensional two-sided ideal}\}.
  \end{equation*}
  If a commutative $\boldsymbol k$-algebra $\{V, m\}$ has a two-sided $d$-dimensional ideal $I$, then
  for every element $g\in G$ the subspace  $g\cdot I$ is a two-sided  $d$-dimensional ideal of $\{V, g\cdot m\}$. Since $G$ acts transitively on the set of all $d$-dimensional linear subspaces of $V$, there exists an element $g$ such that  $g\cdot I$ is the $\boldsymbol k$-linear span
$V_d$ of the vectors $e_1,\ldots, e_d$. Hence
\begin{align}\label{ideal1}
\mathcal I_d&=G\cdot \mathcal L_d,\;\mbox{where}\\
\label{ideal2}
 \mathcal L_d&:=\{m\in\mathcal C\mid
V_d\;\mbox{is the two-sided ideal of $\{V, m\}$}\}.
\end{align}

Consider in $G$ the subgroup
\begin{equation}\label{P}
P_d:=\{g\in G\mid g\cdot V_d=V_d\}.
\end{equation}
It follows from \eqref{ideal2}, \eqref{P} the equality
$P_d\cdot \mathcal L_d=\mathcal L_d,$
which, in view of  \eqref{ideal1} and Lemma \ref{tran},
implies that
\begin{equation}\label{const1}
\dim \mathcal I_d\leqslant \dim G + \dim \mathcal L_d -\dim P_d
\end{equation}
and, since  $P_d$ is a parabolic subgroup of  $G$,
that the set $\mathcal I_d$ is closed in\;$\mathcal C$
(see, e.g.,\;\cite[Sect.\;2.13, Lemma\;2]{St}).

It follows from \eqref{basis} that
$V_d$ is the two-sided ideal of a commutative
 $\boldsymbol k$-algebra $\{V, m\}$, where
\begin{equation}\label{alp}
m=\sum_{\fratop{1\leqslant p\leqslant q\leqslant n}{1\leqslant r\leqslant n}} \alpha_{pq}^r c^{pq}_r,\quad \alpha_{pq}^r \in {\boldsymbol k},
\end{equation}
if and only if  $\alpha_{pq}^r=0$ for all triples $(p, q, r)$ from the range of summation in \eqref{alp}
satisfying the conditions:
\begin{enumerate}[\hskip 2.2mm \rm(i)]
\item $p$ or $q$ is not greater than $d$;
\item $r>d$.
\end{enumerate}
This and \eqref{ideal2} imply that ${\mathcal L}_d$ is a linear subspace in  $\mathcal C$ and,
since in total there are $$(n-d)\bigg(\!\!\binom{n+1}{2}-\binom{n-d+1}{2}\!\!\bigg)=\frac{(n-d)d(2n-d+1)}{2}$$ of these triples, we obtain
\begin{equation}\label{dim}
\dim {\mathcal L}_d
=\dim \mathcal C-\frac{(n-d)d(2n-d+1)}{2}.
\end{equation}

From $\dim G=n^2$, $\dim P_d=n^2-d(n-d)$, and \eqref{dim} we deduce that
the right-hand side of inequality \eqref{const1} is equal to
$\dim\mathcal C-d(n-d)(2n-d+1)/2+d(n-d)$, which, for $d<n$, is smaller than $\dim\mathcal C$.\;Hence $\mathcal C\setminus \mathcal I_d$ is, for every $d<n$, a nonempty open subset of $\mathcal C$.\;Therefore, the intersection of all these open subsets is a nonempty open subset of  $\mathcal C$ such that for every its point  $m$ the $\boldsymbol k$-algebra $\{V, m\}$ is simple.\;This proves the claim of simplicity of algebra in Theorem {\rm \ref{C}}.
\end{proof}

\begin{proof}[Proof of the claim of triviality in Theorem {\rm \ref{C}}]\

1.  As above, since for $n=1$ the claim is clear, we assume further that $n\geqslant 2$.

Since the stabilizer $G_x$ of every point $x\in \mathcal C$ is an algebraic subgroup of $G$,
the claim will be proved if shall prove the existence in $\mathcal C$ the nonempty open subsets
$\mathcal C_{(s)}$ and  $\mathcal C_{(u)}$ such that $G_x$ for every $x\in \mathcal \mathcal C_{(s)}$
(respectively, for every $x\in \mathcal \mathcal C_{(u)}$) does not contain nonidentity semisimple
(respectively, unipotent) elements.

  \vskip 1mm

  2. First we prove the existence of $\mathcal C_{(s)}$.


The plan is the following.\;We shall use that every semisimple element of the group $G$ is conjugate to an element of its fixed maximal torus, see\;\cite[Thms.\;11.10, 10.6, Prop.\;11.19]{Bor}.
Namely, let ${\rm X}(T)$ be the group characters
$T\to \boldsymbol k^*$, and let $\Delta$ be the weight system of the $T$-module $\mathcal C$, i.e.,\;the set of all characters
$\mu\in {\rm X}(T)$ such that the dimension of its weight subspace
 \begin{equation}\label{w}
\mathcal C_\mu:=\{m\in \mathcal C\mid t\cdot m=\mu(t)m\;\;\forall\;
t\in T
\}
 \end{equation}
is positive.\;We have the decomposition
\begin{equation}\label{ww}
\mathcal C=\textstyle\bigoplus_{\mu\in \Delta}\mathcal C_\mu.
\end{equation}
In view of \eqref{w}, \eqref{ww}, for every element $t\in T$, the equality
\begin{equation}\label{www}
\mathcal C^t:=\{m\in \mathcal C\mid t\cdot m=m\}=\textstyle\bigoplus_{\{
\mu\in \Delta\,\mid\, \mu(t)=1
\}}\mathcal C_\mu.
\end{equation}
holds. Since the set $\Delta$ is finite, it follows from \eqref{www} that, when
 $t$ runs through $T$, the space $\mathcal C^t$, being the sum of some of the weight subspaces $\mathcal C_\mu$, runs through only finitely many  $T$-invariant linear subspaces in $\mathcal C$.\;We shall find an integer $h$ such that for every nonidentity $t\in T$ the inequalities
 \begin{align}
 \dim \mathcal C^t&\leqslant h,\label{h1} \\
h+\dim G-\dim T=h+ n^2-n&<\dim \mathcal C=\frac{n^2(n+1)}{2}\label{h2}.
\end{align}
hold.\;In view of Lemma \ref{tran}, it then follows from
 \eqref{h1}, \eqref{h2} that the closure of the set $G\cdot \mathcal C^t$ in $\mathcal C$ is a proper
 subvariety of  $\mathcal C$.\;Since, as is explained above, the set of these closures is finite, the complement to their union is a nonempty open subset of  $\mathcal C$.
From the above it follows that it can be taken as $\mathcal C_{(s)}$.

\vskip 1mm

3. We now proceed to implement this plan.\;As $T$ we take the maximal torus consisting of diagonal matrices. Put
\begin{equation}\label{epsi}
    \varepsilon_i\colon T\to \boldsymbol k^*,\quad
     {\rm diag}(t_1,\ldots, t_n)\mapsto t_i.
\end{equation}
 Then $\varepsilon_1,\ldots, \varepsilon_n$ is a basis of the free abelain group ${\rm X}(T)$ and
\begin{equation}\label{weights}
    t\cdot e_i=\varepsilon_i(t) e_i, \quad t\cdot \ell_i=\varepsilon_i^{-1}(t)\ell_i\quad \mbox{for every element $t\in T$}.
\end{equation}
 Formulas \eqref{weights} and \eqref{basis} imply the inclusion $c^{pq}_r\in \mathcal C^{\ }_{\varepsilon_p^{-1}\varepsilon_q^{-1}\varepsilon_r}$, from which we obtain
\begin{equation}\label{delta}
\left.\begin{split}
\Delta&=\Delta'\cup \Delta'',\;\;\mbox{where}\\
\Delta'&:=\{\varepsilon_p^{-1}\varepsilon_q^{-1}\varepsilon_r\mid 1\leqslant\! p\!\leqslant \!q\!\leqslant n,\; 1\!\leqslant\! r\!\leqslant\! n,\;r\!\neq\! p,\;r\!\neq\! q\},\\
\Delta''&:=\{\varepsilon_s^{-1}\mid 1\!\leqslant\! s\!\leqslant\! n\},\;\;\mbox{and}\\
\dim \mathcal C^{\ }_{\mu}&=
\begin{cases} 1&\mbox{if $\mu\in \Delta'$},\\
n&\mbox{if $\mu\in \Delta''$}.
\end{cases}
\end{split}\right\}
\end{equation}

  Every character $\alpha\in{\rm X}(T)$ defines a partition of the set  $\Delta$
  as the union of mutually
  disjoint sequences of weights
  $\mu_1,\ldots,\mu_d\in\Delta$ having the properties
 \begin{enumerate}[\hskip 2.2mm \rm(i)]
 \item $\mu_{i+1}=\alpha\mu_i$ for every $i=1,\ldots, d-1$,
 \item $\alpha^{-1}\mu_1\notin \Delta$ and $\alpha\mu_d\notin\Delta$.
 \end{enumerate}
We call such sequences $\alpha$-{\it series}.

The kernel of action of  $G$ on $\mathcal C$ is trivial: being a normal subgroup, it lies in the center of
 $G$, hence also in  $T$, and it follows from \eqref{epsi}, \eqref{weights},  and the inclusion $\Delta''\subset \Delta$ that the kernel of action of  $T$ on $\mathcal C$ is trivial.

 Fix in $T$ a nonidentity element $t={\rm diag}(t_1,\ldots, t_n)$.

   If $t_1=\ldots=t_n\neq 1$, then it follows from \eqref{delta} that $\mathcal C^t=\{0\}$, so $\mathcal C\setminus \mathcal C^t$ is a nonempty open subset of $\mathcal C$.

 Now let  $t_s\neq t_{s+1}$ for some  $s$, or, in other words,
 \begin{equation}\label{root}
 \alpha_s(t)\neq 1,\quad \mbox{where}\quad \alpha_s:=\varepsilon_s\varepsilon_{s+1}^{-1}.
 \end{equation}
  It follows from \eqref{delta} and the definition of $\alpha$-series that all  sequences of  forms \eqref{1}--\eqref{6} listed below are the  $\alpha_s$-series for every $p$, $q$, and $r$ satisfying the specified restrictions:
\begin{align}
& \varepsilon_s^2\varepsilon_{s+1},\; \varepsilon_s^{-1},\; \varepsilon_{s+1}^{-1},\; \varepsilon_{s+1}^{-2}\varepsilon_s;\label{1}\\
&\varepsilon_p^{-1}\varepsilon_s^{-1}\varepsilon_{s+1},\;\varepsilon_p^{-1},\;\varepsilon_j^{-1}\varepsilon_{s+1}^{-1}\varepsilon_{s},\quad p\neq s,  s+1;\label{2}\\
&\varepsilon_p^{-1}\varepsilon_q^{-1}\varepsilon_{r},\quad r\neq p, q\;\mbox{and among $p, q, r$ there is no $s$ and $s+1$};\label{3}\\
& \varepsilon_p^{-1}\varepsilon_q^{-1}\varepsilon_{s+1},\; \varepsilon_p^{-1}\varepsilon_q^{-1}\varepsilon_{s},\quad p\neq s, s+1,\;q\neq s, s+1;\label{4} \\
& \varepsilon_s^{-1}\varepsilon_q^{-1}\varepsilon_{r},\; \varepsilon_{s+1}^{-1}\varepsilon_q^{-1}\varepsilon_{r},\quad
r\neq s, s+1,\;q\neq s, s+1,\;r\neq q;\label{5}\\
& \varepsilon_s^{-2}\varepsilon_{r},\; \varepsilon_s^{-1}\varepsilon_{s+1}^{-1}\varepsilon_r,\; \varepsilon_{s+1}^{-2}\varepsilon_r, \quad r\neq s, s+1.\label{6}
 \end{align}

Since for every weight from $\Delta$ there exists only one of the  $\alpha_s$-series of the forms  \eqref{1}--\eqref{6} in which it lies, there are no other $\alpha_s$-series.
 The number of different $\alpha_s$-series of every form  \eqref{1}--\eqref{6} is given in the following table where in the top raw are listed the forms of  $\alpha_s$-series and in the bottom the specified numbers:

 {\fontsize{9pt}{6mm}\selectfont
\begin{center}
\begin{tabular}{c|c|c|c|c|c}
\eqref{1}&\eqref{2}&\eqref{3}& \eqref{4}&
\eqref{5}&
\eqref{6}
\\
\hline \hline
&&&&&\\[-13pt]
$1$& $n-2$&$(n-2)^2(n-3)/2$&$(n-1)(n-2)/2$&$(n-2)(n-3)$&$n-2$\\
\end{tabular}
\end{center}
}

  It follows from \eqref{root} and property (i) in the definition of  $\alpha$-series that $t$ can not lie in  the intersection of kernels of any two neighboring weights of every $\alpha_s$-series. Therefore, the set of weights of any $\alpha_s$-series of form \eqref{1} that contains  $t$ in its kernel is either empty, or contains only one weight,
or exactly two non-neighboring weights.\;In view of \eqref{delta} this shows that the sum of dimensions of the weight spaces of this set is not greater than  $n+1$.

Similarly, we obtain that if $\mu$ runs through all the weights of some $\alpha_s$-series of form
\eqref{2}, \eqref{3}, \eqref{4}, \eqref{5}, or \eqref{6} that contains $t$ in its kernel, then the number
 $\sum \dim\mathcal C_\mu$ is not bigger than, respectively,  $n$, $1$, $1$, $1$, or $2$ (for form \eqref{2}, we use here that $n\geqslant 2$).

It follows from this and the above table that
 if  $\mu$ runs through {\it all} $\alpha_s$-series of form  \eqref{1},
\eqref{2}, \eqref{3}, \eqref{4}, \eqref{5}, or \eqref{6} that contains $t$ in its kernel, then the number
 $\sum \dim\mathcal C_\mu$ is not greater than, respectively,
 $n+1$, $n(n-2)$, $(n-2)^2(n-3)/2$, $(n-1)(n-2)/2$, $(n-2)(n-3)$, or $2(n-2)$.

 Since $\Delta$ is the union of all $\alpha_s$-series, this shows that the inequality
 \eqref{h1} holds for
 \begin{equation}\label{hhh}
 \begin{split}
 h=n&+1+n(n-2)+
 \frac{(n-2)^2(n-3)}{2}\\
 &+\frac{(n-1)(n-2)}{2}
 +(n-2)(n-3)+2(n-2).
 \end{split}
 \end{equation}

 A simple direct calculation then shows that inequality  \eqref{h2} follows from \eqref{hhh}.\;This completes the proof of the existence of $\mathcal C_{(s)}$.

\vskip 1mm

4.  We now turn to the proof of the existence of $\mathcal C_{(u)}$.

First, we prove that the stabilizers $G_m$ of points $m$ in general position in $\mathcal C$ are finite,  or, equivalently, that the automorphism groups of the  $\boldsymbol k$-algebras $\{V, m\}$  are finite.\;Since the set of points of  $\mathcal C$, whose  $G$-stabilizer has minimal dimension, is open in   $\mathcal C$, in order
to prove the finiteness it suffices to find a single multiplication $m_0\in\mathcal C$ such that the automorphism group of the $\boldsymbol k$-algebra $\{V, m_0\}$ is finite. Note, however, that from the triviality of the automorphism group of the
$\boldsymbol k$-algebra $\{V, m_0\}$ we can {\it not} conclude that the automorphism group of the
$\boldsymbol k$-algebra $\{V, m\}$ is trivial for points $m$ in general position in $\mathcal C$,
see\;\cite[Sect.\;6.1, Example\;$1$]{PV}.

Define $m_0\in \mathcal C$ by the following multiplication table:
\begin{equation}\label{table}
e_ie_j:=\begin{cases}e_i&\mbox{if $j=i$},\\
0&\mbox{if $j\neq i$.}
\end{cases}\quad
\mbox{for all $1\leqslant i, j,\leqslant n$.}
\end{equation}
Then $A:=\{V, m_0\}$ is the direct sum of one-dimensional subalgebras:
$A=A_1\oplus \cdots\oplus A_n$, where $e_i$ is the basis of $A_i$, and
the multiplication is given by the formula $e_i^2=e_i$ showing that $e_i$ is the identity in the $\boldsymbol k$-algebra $A_i$.

All one-dimensional two-sided ideals $I$ of $A$ are
exhausted by the ideals $A_1,\ldots, A_n$.\;Indeed, let a vector
$v=\alpha_1e_1+\cdots+\alpha_ne_n$, where $\alpha_1,\ldots, \alpha_n\in\boldsymbol k$, be a basis of $I$, and let
$\alpha_p\neq 0$.\;Then $e_pv=\alpha_pe_p$ in view of
\eqref{table}. On the other hand, since $I$ is a one-dimensional two-sided ideal,  $e_pv=\lambda v$ for some $\lambda\in\boldsymbol k$.
 It then follows from $\alpha_p\neq 0$ that $\lambda\neq 0$, hence $v=\alpha_p\lambda^{-1}e_p$.\;Therefore, $I=A_p$.

Now let $\sigma$ be an automorphism of  the $\boldsymbol k$-algebra $A$.\;Then $\sigma(A_i)$
 for every $i$ is a one-dimensional two-sided ideal of  $A$, hence, as is proved, it coincides with $A_j$ for some  $j$ and the restriction of $\sigma$ to $A_i$ is an isomorphism of $\boldsymbol k$-algebras $A_i\to A_j$. Since isomorphism maps identity to identity, $\sigma(e_i)=e_j$. Thus the automorphism group of the $\boldsymbol k$-algebra $A$ leaves invariant the finite set $\{e_1,\ldots, e_n\}$. Since $e_1,\ldots, e_n$ is a basis in $V$,
its action on this set has no kernel.\;Therefore, this group is finite\footnote{Its order is $n!$ since the element of $G$ inducing a permutation of $e_1,\ldots, e_n$ is an automorphism of the $\boldsymbol k$-algebra\;$A$.}.

\vskip 1mm

5. If ${\rm char}\,\boldsymbol k=0$, then the order of every nonidentity unipotent element is infinite, therefore,
the existence of  $\mathcal C_{(u)}$ follows from the proved finiteness of the stabilizers of points in general position in $\mathcal C$.\;In the case of ${\rm char}\,\boldsymbol k>0$, the order of every unipotent element is finite, therefore, this argument is not applicable. Consider this case.

  Let $u\in G$ be a nonidentity unipotent element, let $C$ be the finite cyclic group that it generates, and let  $\boldsymbol kC$ be its group algebra over $\boldsymbol k$. The order of $C$ is a positive power ${\rm char}\,\boldsymbol k$.\;It follows from the Jordan normal form theory (see also
\cite{Hi}), that for every positive integer $d$ not bigger than $C$, there is a unique up to isomorphism indecomposable $d$-dimensional $\boldsymbol kC$-module $M_d$, and there are no other indecomposable $\boldsymbol kC$-modules.\;In particular, $I:=M_1$ is a one-dimensional trivial $\boldsymbol kC$-module.\;The said implies that  $M_d$ is isomorphic to its dual module  $M_d^*$.\;In a certain basis, the matrix of the linear operator
 $M_d\to M_d$, $m\mapsto u\cdot m$ is the Jordan block with eigenvalue $1$,
 and therefore,
\begin{equation}\label{fix}
\dim_{\boldsymbol k} M_d^C=1.
\end{equation}

If the $\boldsymbol kC$-modules $P$ and $Q$ are isomorphic, we shall write
$P\simeq Q$.\;If $N$ is a finite dimensional $\boldsymbol kC$-module, we denote by $N^{\oplus r}$
 and $N^{\otimes r}$  respectively the direct sum and the tensor product of  $r>0$ copies of the $\boldsymbol kC$-module $N$; we put   $N^{\otimes 0}:=I^{\oplus \dim_{\boldsymbol k} N}$.  By
$|N|$ we denote the number of indecomposable  $\boldsymbol kC$-submodules,
whose direct sum is $N$ (by the Krull--Schmidt theorem this number is well-defined).\;According to \cite[Lemma 2.1]{Ra},
\begin{equation}\label{min}
|M_p\otimes M_q|={\rm min}\{p, q\}\;\;\mbox{for every $p$ and $q$.}
\end{equation}
\begin{lemma}\label{j}
For every finite dimensional   $\boldsymbol kC$-modules $P$ and $Q$  and every positive integer $d$, the following inequality holds:
\begin{align}\label{1st}
|P\otimes Q|&\leqslant |P\otimes I^{\oplus \dim_{\boldsymbol k} Q}|.
\end{align}
\end{lemma}
\begin{proof}[Proof of Lemma {\rm\ref{j}}] Decompose $P$ and $Q$ as the direct sums of indecomposable $\boldsymbol kC$-submodules:
\begin{equation}\label{dec}
P\simeq \textstyle\bigoplus_p M_{s_p},\quad Q\simeq \textstyle\bigoplus_q M_{r_q}.
\end{equation}
From these decompositions we obtain
\begin{align*}\label{decc}
|P\otimes Q|&\overset{\eqref{dec}}{=}|\textstyle\bigoplus_{p, q} \big(M_{s_p}\otimes M_{r_q}\big)|=\textstyle \sum_{p, q} |M_{s_p}\otimes M_{r_q}|\\[-5pt]
&\overset{\eqref{min}}{\leqslant}\textstyle \sum_{p, q} |M_{s_p}\otimes I^{\oplus {r_q}}|\!=\!
|\textstyle\bigoplus_{p, q}(M_{s_p}\otimes I^{\oplus {r_q}})|\!
\overset{\eqref{dec}}{=}\!|P\otimes I^{\oplus \dim_{\boldsymbol k} Q}|.
\end{align*}
This proves Lemma \ref{j}.\end{proof}

Since $u\in G$, and $V$ and $\mathcal C$ are the ${\boldsymbol k}G$-modules, then
 $V$ and $\mathcal C$ are also ${\boldsymbol k}C$-modules. Since $u$ is a nonidentity and the kernel of the ${\boldsymbol k}G$-module $V$ is trivial, the decomposition of the ${\boldsymbol k}C$-module $V$ as
 a direct sum of the indecomposable submodules contains at least one summand of dimension $s\geqslant 2$. Therefore, for some ${\boldsymbol k}C$-module $N$ we have
\begin{equation}\label{decV}
V\simeq M_s\oplus N, \quad s\geqslant 2.
\end{equation}
Since the  ${\boldsymbol k}G$-modules $\mathcal C$ and
$({\rm Sym}^2V^*)\otimes V$ are isomorphic, from \eqref{decV}
and the self-duality of finite dimensional ${\boldsymbol k}C$-modules
we obtain
\begin{equation}\label{CC}
\mathcal C\simeq  ({\rm Sym}^2(M_s\oplus N))\otimes (M_s\oplus N).
\end{equation}
Using that  ${\rm Sym}^2(A\oplus B)$ and $({\rm Sym}^2A)\oplus (A\otimes B)\oplus ({\rm Sym}^2B)$
are isomorphic for every modules $A$ and $B$ (see, e.g.,\;\cite[(B.2), p.\;473]{FH}),
from \eqref{CC} we obtain
\begin{equation}\label{CCC}
\begin{split}
\mathcal C\simeq\; & (({\rm Sym}^2M_s)\otimes M_s)\oplus (M_s^{\otimes 2}\otimes N)\oplus (({\rm Sym}^2N)\otimes M_s)\\
&\quad\oplus (({\rm Sym}^2M_s)\otimes N)\oplus (N^{\otimes 2}\otimes M_s)\oplus (({\rm Sym}^2N)\otimes N).
\end{split}
\end{equation}
Note also that ${\rm Sym}^2M_s\varsubsetneq M_s^{\otimes 2}$ and \eqref{min} imply the inequality
\begin{equation}\label{-1}
|{\rm Sym}^2M_s|\leqslant s-1.
\end{equation}
From Lemma \ref{j}, taking into account \eqref{-1}, \eqref{min} and the definition of $|\cdot|$, we obtain
\begin{equation}\label{CCCC}
\left.\begin{split}
|({\rm Sym}^2M_s)\otimes M_s|&\overset{\eqref{1st}}{\leqslant}|({\rm Sym}^2M_s)\otimes I^{\oplus s}|\overset{\eqref{-1}}{\leqslant} s(s-1),\\
|M_s^{\otimes 2}\otimes N|&\overset{\eqref{1st}}{\leqslant}
|M_s^{\otimes 2}\otimes I^{\oplus (n-s)}|\overset{\eqref{min}}{=}s(n-s),\\
|({\rm Sym}^2N)\!\otimes\! M_s|&\!\overset{\eqref{1st}}{\leqslant}\!
\big|M_s\!\otimes\! I^{\oplus \frac{(n-s+1)(n-s)}{2}}\big|\!=\!\frac{(n\!-\!s\!+\!1)(n\!-\!s)}{2},\\
|({\rm Sym}^2M_s)\otimes N|&\overset{\eqref{1st}}{\leqslant}
|({\rm Sym}^2M_s)\otimes I^{\oplus (n-s)}|\!\overset{\eqref{-1}}{\leqslant}\!(n\!-\!s)(s\!-\!1),\\
|N^{\otimes 2}\otimes M_s|&\overset{\eqref{1st}}{\leqslant}
|M_s\otimes I^{\oplus (n-s)^2}|=(n-s)^2,\\
|({\rm Sym}^2N)\otimes N|&\hskip 2mm{\leqslant}\dim(({\rm Sym}^2N)\otimes N)\!=\!\frac{(n\!-\!s\!+\!1)(n\!-\!s)^2}{2}.
\end{split}\right\}
\end{equation}
It follows from \eqref{CCC} and \eqref{CCCC} that
\begin{equation}\label{di}
\begin{split}
|\mathcal C|\leqslant s(s-1)&+s(n-s)+\frac{(n-s+1)(n-s)}{2}+(n-s)(s-1)\\[-2pt]
&+(n-s)^2+\frac{(n-s+1)(n-s)^2}{2}.
\end{split}
\end{equation}

  But \eqref{fix} implies that for every finite dimensional $\boldsymbol k C$-module $N$ the equality $N^C=|N|$ holds.\;Therefore, \eqref{di} means that $\dim_{\boldsymbol k}\mathcal C^u$ is not bigger than the right-hand side of inequality \eqref{di}.\;As is known, dimension of the centralizer of the element $u$ in $G$ is not smaller than the rank of the group $G$ that is equal to $n$ (see, e.g.,\;\cite[Prop.\,1,\;p.\,94]{St}).\;Since this centralizer lies in  $N_G(\mathcal C^u)$, it now follows from Lemma \ref{tran} that dimension of the closure $\overline{G\cdot \mathcal C^u}$ of the set $G\cdot \mathcal C^u$ in $\mathcal C$ is not bigger than the number
  \begin{equation}\label{numb}
\begin{split}
s(s-1)&+s(n-s)+\frac{(n-s+1)(n-s)}{2}+(n-s)(s-1)\\[-2pt]
&+(n-s)^2+\frac{(n-s+1)(n-s)^2}{2}+n^2-n.
\end{split}
\end{equation}

Twice the difference between $\dim\mathcal C=(n^3+n^2)/2$ and number \eqref{numb} is equal to
\begin{equation}\label{dif}
n^2(3s-5)+n(-3s^2+4s+3)+(s^3-2s^2+s).
\end{equation}
Consider \eqref{dif} is the polynomial in $n$.\;Its leading coefficient is positive for every
$s\geqslant 2$, and the discriminant is equal to
$-3s^3(s-20/3)-54s(s-22/27)+9$, which shows that the latter is negative
for every $s\geqslant 7$. Direct calculation shows that it is negative also for $s=2, 3, 4, 5, 6$.
Therefore, for every $s\geqslant 2$ the specified difference is positive.

Hence $\mathcal C\setminus\overline{G\cdot \mathcal C^u}$ is a nonempty open subset of $\mathcal C$ such that the $G$-stabilizers of its points do not contain elements conjugate to $u$.\;But there are only finitely many conjugacy classes of unipotent elements in $G$ in view of the finiteness of the set of possible Jordan normal forms of such elements (see also \cite{St}).\;Therefore, if $u$ runs through the set of representatives of nonidentity conjugacy classes of unipotent elements in  $G$, then the intersection of the open subsets
$\mathcal C\setminus\overline{G\cdot \mathcal C^u}$ is a nonempty open subsubset of $\mathcal C$.\;The said implies that this subset may be taken as\;$\mathcal C_{(u)}$.

This completes the proof of Theorem {\rm \ref{C}}.
\end{proof}

\begin{remarknonumber} Actually, Parts 2 and 3 of the proof give more than we used. Namely,
the arguments used there show that for every constant $\alpha\in \boldsymbol k$ and every element $t\in T$
 which does not lie in the center of the group $G$ (i.e., which is not scalar), inequalitites  \eqref{h1} and \eqref{h2}, where $h$ is given by formula \eqref{hhh}, still hold if  $\mathcal C^t$ in \eqref{h1} is replaced by $\{m\in \mathcal C\mid t\cdot m=\alpha m\}$.\;This means that we have proved the following
\begin{theorem}
For the natural action of the group $G$ on the projectivization $P\mathcal C$ of the vector space
$\mathcal C$, the stabilizers of points in general position in $P\mathcal C$ coincide with the center of the group $G$.
\end{theorem}
\end{remarknonumber}

\begin{proof}[Proof of Theorem {\rm\ref{M}}] \

The claim of simplicity is proved in \cite[Thm.\;4]{Po},
 and that of triviality follows from Theorem
 \ref{C} in view of the first equality in \eqref{decomp}.
 \end{proof}

\section{Proofs of Lemma \ref{123}, Theorems \ref{rat}, \ref{section}, {\rm \ref{sectionC}}, and Corollary \ref{sr}}


The proof of Lemma \ref{123} is based on one frequently used statement known as the No-Name Lemma. We recall its formulation.

 Let $\pi\colon E\to X$ be an algebraic vector bundle over a variety $X$. Assume that
  $E$ and $X$ are endowed with the actions of an affine algebraic group $G$ such that $\pi$ is the $G$-equivariant morphism and for every elements $g\in G$, $x\in X$ the map of fibers $\pi^{-1}(x)\to \pi^{-1}(g\cdot x)$ defined by the transformation $g$ is $\boldsymbol k$-linear.

 \begin{lemma}[{{\rm No-Name Lemma}}] \label{Nnl} Using the above notation, assume that the action of $G$ on $X$ is locally free and put
 $d=\dim E-\dim X$. Consider the action of  $G$ on $X\times\bAd$ via the first factor and let $\pi_1\colon X\times \bAd\to X$ be the natural projection.\;Then there exists a $G$-equivariant birational isomorphism  $\varphi\colon E\dashrightarrow X\times\bAd$ such that the following dia\-gram is commutative
\begin{equation*}
\begin{matrix}
\xymatrix@=6mm{E\ar[dr]_{\pi}\ar@{-->}[rr]^{\varphi}&&X\times\bAd\ar[dl]^{\pi_1}\\
&X&
}
\end{matrix}.
\end{equation*}
 \end{lemma}
\begin{proof} If the group $G$ is finite, it is the classical ``Speiser Lemma''\;\cite{Sp};
in the general case, the proof, valid in arbitrary characte\-ristic, see, e.g., in \cite{RV}.
\end{proof}

From now on we use the following notation.\;If  $X$ and $Y$ are the irreducible varieties, then $X\approx Y$ means that  $X$ and $Y$ are birationally isomorphic; if an algebraic group $H$ acts on  $X$ and $Y$, then $X\overset{H}{\approx} Y$ means that there is an $H$-equivariant birational isomorphism between $X$ and $Y$.

 \begin{proof}[Proof of Lemma {\rm \ref{123}}] \

 Put $d_i:=\dim L_i$.
 The natural projection $L_2\oplus L_3\to L_2$ is a vector bundle over $L_2$, to which, in view of
condition   (i), is applicable Lemma  \ref{Nnl}.\;It implies that
\begin{equation}\label{23}
L_2\oplus L_3 \overset{H}{\approx} L_2\times {\mathbf A}\!^{d_3},
\end{equation}
where $H$ acts on $L_2\times {\mathbf A}\!^{d_3}$ via the first factor.\;Similarly, applying Lemma
 \ref{Nnl} to the natural projections $L_1\oplus L_2\to L_1$ and $L_1\oplus L_2\to L_2$
and considering the actions of $H$ on $L_1\times {\mathbf A}\!^{d_2}$ and $L_2\times {\mathbf A}\!^{d_1}$ via the first factors, we obtain
\begin{equation}\label{12}
L_1\times {\mathbf A}\!^{d_2} \overset{H}{\approx} L_1\oplus L_2 \overset{H}{\approx} L_2\times {\mathbf A}\!^{d_1}.
\end{equation}

It follows from \eqref{23} and condition (iii) that
\begin{equation}\label{q1}
(L_2\oplus L_3)\dss H\approx (L_2\dss H)\times  {\mathbf A}\!^{d_3}= (L_2\dss H)\times  {\mathbf A}\!^{d_1}\times {\mathbf A}\!^{d_3-d_1},
\end{equation}
and from \eqref{12} that
\begin{equation}\label{q2}
 (L_1\dss H)\times  {\mathbf A}\!^{d_2} \approx (L_1\oplus L_2)\dss H\approx (L_2\dss H)\times  {\mathbf A}\!^{d_1}.
\end{equation}
In view of condition (ii), we have  $(L_1\dss H)\approx {\mathbf A}\!^{\dim L_1\dss H}$,
from where, in view of
 \eqref{q1} and \eqref{q2}, we obtain
$(L_2\oplus L_3)\dss H\approx {\mathbf A}\!^{d_2+d_3-d_1+\dim L_1\dss H}$.
This completes the proof.
\end{proof}

\begin{proof}[\it Proof of Theorem {\rm \ref{rat}}] \

 The claim is clear for $n=1$, so further we assume that $n\geqslant 2$.

   Let $V^{\oplus n}$ be the direct sum of  $n$ copies of the space $V$.\;One of the orbits of the diagonal action of  $G$ on $V^{\oplus n}$ is open in $V^{\oplus n}$, therefore, ${\boldsymbol k}(V^{\oplus n})^G={\boldsymbol k}$.\;It is not difficult to see that this action is locally free.\;It follows from this, \eqref{dimA}, and Theorem  \ref{C} that for
 $H=G$, $L_1=V^{\oplus n}$, $L_2=\mathcal C$, $L_3=\mathcal A$, and $n\geqslant 3$
 the conditions of Lemma \ref{123} hold (for $n=2$ condition (iii) of this lemma does not hold).\;Hence for $n\geqslant 3$ the claim  that we are proving follows from
 \eqref{decomp} and Lemma \ref{123}.

  All multiplications $m$ from the $G$-module $\mathcal C$ (respectively, $\mathcal A$) such that, for every element $v\in \{V, m\}$, the linear operator $V\to V$, $a\mapsto va$ has zero trace, constitute a submodule $\mathcal C_0$ (respectively, $\mathcal A_0$).\;Besides, every element $\ell\in V^*$ defines the multiplications $m_{\ell +}\in\mathcal C$, $m_{\ell -}\in\mathcal A$, for which  the products of elements $a, b\in V$ are given, respectively, by the formulas
 \begin{align}
 ab&:=\ell(a)b+\ell(b)a,\label{+}\\
 ab&:=\ell(a)b-\ell(b)a.\label{-}
 \end{align}
The subsets $\mathcal C':=\{m_{\ell +}\mid \ell\in V^*\}$ and $\mathcal A':=\{m_{\ell -}\mid \ell\in V^*\}$
also are the submodules of, respectively, the $G$-modules $\mathcal C$ and $\mathcal A$.\;Both of these submodules are isomorphic to the $G$-module $V^*$.\;For every of the $\boldsymbol k$-algebras $\{V, m_{\ell +}\}$ and $\{V, m_{\ell +}\}$, the trace of the operator of left multiplication by an element $v\in V$ is equal to $n\ell(v)$.
This yields
\begin{equation}\label{chara}
\mathcal C=\mathcal C_0\oplus \mathcal C'\;\;\mbox{and}\;\;\mathcal A=\mathcal A_0\oplus \mathcal A'\;\;\mbox{if ${\rm char}\,\boldsymbol k$ does not divide  $n$.}
\end{equation}

 Now let $n=2$. Then  $\mathcal A=\mathcal A'$ in view of \eqref{dimA}, and it follows from \eqref{decomp},  \eqref{chara}, and the condition ${\rm char}\,{\boldsymbol k}\neq 2$ that
\begin{equation}\label{n2}
\mathcal M=\mathcal C_0\oplus \mathcal C'\oplus \mathcal A'.
\end{equation}
 Since one of the $G$-orbits is open in  $V^*\oplus V^*$, we have $\boldsymbol k(V^*\oplus V^*)^G=\boldsymbol k$. It is not difficult to see that the action of  the group $G$ on $V^*\oplus V^*$ is locally free.\;This implies that for $H=G$, $L_1=L_2=\mathcal C'\oplus\mathcal A'$, and $L_3=\mathcal C_0$ the conditions of Lemma \ref{123} hold. Therefore, for $n=2$, the claim under the proof follows from this lemma and  \eqref{n2}.
\end{proof}

The proofs of
 Theorems \ref{section} and \ref{sectionC} are based on the
 following general statement (Theorem \ref{SS} below) about {\it special}  algebraic groups
 in the sense of Serre.\;Recall from \cite[Sect.\;4.1]{Se} that these are algebraic groups $S$ such that every $S$-torsor is trivial in the Zariski topology (or, equivalently, for every field extension $K/\boldsymbol k$ the Galois cohomology $H^1(K, S)$ is trivial).
\begin{theorem}\label{SS}
 Let $X$ be an irreducible variety endowed with a locally free action of a special algebraic group $S$.\;Then there exists an irreducible closed subset $Z$ of $X$ such that:
\begin{enumerate}[\hskip 2.2mm\rm(i)]
\item the map $\boldsymbol k(\mathcal M)^G\to \boldsymbol k(\mathcal S)$, $f\mapsto f|_{\mathcal S}$ is well-defined and is a $\boldsymbol k$-isomorphism of fields;
    \item
for every point  $x$ of an open dense subset of $X$, there exists a unique element $s\!\in\! S$ such that
$s\cdot x\!\in\!\mathcal S$.
\end{enumerate}
\end{theorem}

\begin{proof} Consider a rational quotient  $\pi_{S, X}\colon X\dashrightarrow X\dss S$. By Rosenlicht's theorem  \cite[Thm.\,2]{Ros} there exists an $S$-invariant open subset $U$ of $X$ lying in the domain of definition of $\pi_{S, X}$ such that $W:=\pi_{S, X}(U)$ is open in $\mathcal X\dss S$ and every fiber of the morphism
\begin{equation}
\label{torsor}
\pi_{S, X}|_U\colon U\to W
\end{equation}
is an $S$-orbit.
Since the action is locally free, replacing $U$ by a suitable invariant open subset, we may assume that
for every point $x\in U$, the orbit map $S\to S\cdot x$, $s\mapsto s\cdot x$ is an isomorphism.\;In turn, this implies
(see\;\cite[Remark\;4]{RV}) that making another such a replacement, we may assume that \eqref{torsor} is a torsor.\;It then follows from the specialness of the group $G$ that morphism \eqref{torsor}, and hence the rational map $\pi_{S, X}$, admit a rational section
$\sigma\colon X\dss S\dashrightarrow X$ (i.e.,\;$\pi\circ\sigma={\rm id}$).\;This implies that one can take the closure  of the set $\sigma(X\dss S)$ in $X$ as $Z$ from the formulation of Theorem\;\ref{section}.
\end{proof}

In the proof of Theorems {\rm \ref{section}} and {\rm \ref{sectionC}} we shall consider the Lie algebra
 $\mathfrak g$ of the group $G$. It is the Lie algebra
 $\mathfrak{gl}(V)$ of all $\boldsymbol k$-linear operators on
$V$, see\;\cite[I.3.1, Example\;(2)]{Bor}.\;For the representation of $G$ on $V$ under consideration, the differential at the identity of $G$
is the identity representation of the Lie algebra
$\mathfrak{gl}(V)$ on $V$.

\begin{proof}[\it Proofs of Theorems {\rm \ref{section}} and {\rm \ref{sectionC}}]\

We shall prove that the actions of the group $G$ on $\mathcal M$ and $\mathcal C$ are locally free.\;Since the group $G$
is special (see\;\cite[\S4]{Se}), Theorems
 {\rm \ref{section}} and\;{\rm \ref{sectionC}} will then follow from Theorem \ref{SS}.

In view of Theorem\;\ref{C} (respectively, Theorem\;\ref{M}), for the points $m$ in general position if
$\,\mathcal C$ (respectively, in $\mathcal M$)  the stabilizers  $G_m$  are trivial.\;Therefore,
the proof of local freeness of the actions under consideration reduces to proving that for such $m$
the orbit morphism $G\to G\cdot m,\;g\mapsto g\cdot m$ is separable.
Its separability is equivalent to surjectivity of the differential of this morphism at the identity,
see\,\cite[AG.17.3, II.6.7]{Bor}.\;Since the image of this differential is the linear subspace
$\mathfrak g\cdot m$ of $\,\mathcal C$ (respectively, of $\mathcal M$), and the orbit $G\cdot m$
is a $\dim G$-dimensional (because of triviality of  $G_m$)
smooth variety, the problem reduces to proving the equality
$\dim \mathfrak g\cdot m=\dim G$ for the specified $m$.

We now observe that, for every
 $d\in \mathbb Z$, the  (possibly, empty) subset
 $\{m\in \mathcal M\mid \dim \mathfrak g\cdot m\geqslant d\}$  of $\mathcal M$
 is open.\;This follows from the fact that, in the coordinate form, the condition of belonging of a point
  $m$ to the specified subset is expressed as nonvanishing of at least one of the minors of a fixed size of the corresponding matrix.

The said reduces the problem to proving
the existence of a {\it single} point
$c\in \mathcal C$ such that
\begin{equation}\label{aim}
\dim \mathfrak g\cdot c=\dim G
\end{equation}
(note that in this case the stabilizer
$G_c$ is finite, but not necessarily trivial). We now show that it is possible to take
\begin{equation}\label{ide}
c:=\ell^1\otimes\ell^1\otimes e_1+\cdots+\ell^n\otimes\ell^n\otimes e_n
\end{equation}
(see the notation introduced in the beginning of Section 2). Note that $\{V, c\}$ is the direct sum of the one-dimensional algebras with nonzero multiplication considered above in the proof of
the claim of triviality in Theorem \ref{C} (where the notation $m_0$ is used in place of $c$).

For every two positive integers $i$ and $j$ not exceeding $n$, denote by
$x_{i,j}$ the element of $\mathfrak g$ defined by the formula
\begin{equation}\label{xij}
x_{i,j}\cdot e_s= \left\{\!\!\begin{array}{ll}
e_j,&\mbox{если $s=i$},\\
0,&\mbox{если $s\neq i$.}
\end{array}\right.
\end{equation}
 It follows from \eqref{xij} and the definition of the representation of
  $\mathfrak g$ on $V^*$ that
\begin{equation}\label{xij*}
x_{i,j}\cdot \ell_s= \left\{\!\!\begin{array}{ll}
-\ell_i,&\mbox{если $s=j$},\\
0,&\mbox{если $s\neq j$.}
\end{array}\right.
\end{equation}
After simple calculations based on  \eqref{ide}, \eqref{xij} и \eqref{xij*} we obtain
\begin{equation}\label{final}
x_{ij}\cdot c= \ell_i\otimes\ell_i\otimes e_j-\ell_i\otimes\ell_j\otimes e_j-\ell_j\otimes\ell_i\otimes e_j.
\end{equation}

 It follows from \eqref{final} that all the vectors $x_{ij}\cdot c$ are linearly independent over $\boldsymbol k$. Since all the elements $x_{i,j}$ are also linearly independent over $\boldsymbol k$ and their linear span over  $\boldsymbol k$ is $\mathfrak g$, this proves equality \eqref{aim} and completes the proofs of Theorems \ref{section} and \ref{sectionC}.
\end{proof}

\begin{proof}[Proof of Corollary {\rm \ref{sr}}] \

This follows from the known fact (see, e.g.,\;\cite[Cor.\;2(i)]{Po2}) that if $X$ in Theorem \ref{SS}
is an affine space and the action is linear, then the field ${\boldsymbol k}(X)^S$ is stably rational
over $\boldsymbol k$ (indeed, $X$ is birationally isomorphic to $G\times (X\dss G)$ by Theorem  \ref{SS}, therefore, the claim on stable rationality follows from rationality of the underlying variety of $G$).
\end{proof}

\end{document}